\documentclass{article}

\usepackage[english]{babel}
\usepackage[utf8]{inputenc}
\usepackage{amsmath, amsthm}
\usepackage{authblk}
\usepackage{graphicx}
\usepackage[letterpaper, hmargin=1.5in, vmargin=1in]{geometry}
\usepackage[textsize=scriptsize,colorinlistoftodos]{todonotes}
\usepackage{hyperref}
\hypersetup{
    colorlinks,
    linkcolor={red!60!black},
    citecolor={blue!60!black},
    urlcolor={blue!80!black}
}
\usepackage[capitalize, noabbrev]{cleveref}
\usepackage{algorithm}
\usepackage[noend]{algpseudocode}
\usepackage{comment}
\usepackage{booktabs}
\usepackage[margin=30pt,font=small,labelfont=bf,skip=0pt]{caption}
\usepackage{tikz}

\algnewcommand\algorithmicforeach{\textbf{for each}}
\algdef{S}[FOR]{ForEach}[1]{\algorithmicforeach\ #1\ \algorithmicdo}

\newtheorem{theorem}{Theorem}
\newtheorem{definition}[theorem]{Definition}
\newtheorem*{problem}{Main Problem}
\theoremstyle{definition}
\newtheorem*{example}{Example}

\DeclareMathOperator{\cut}{cut}
\DeclareMathOperator{\Ncut}{Ncut}
\DeclareMathOperator{\assoc}{assoc}
\DeclareMathOperator{\Nassoc}{Nassoc}

\title{Finding Minimum Spanning Forests in a Graph}

\author{Abdel-Rahman Madkour, Phillip Nadolny, Matthew Wright}
\affil{St.\ Olaf College, Northfield, MN USA}

\date{\today}

\begin{document}
\maketitle

\begin{abstract}
  We introduce a graph partitioning problem motivated by computational topology and propose two algorithms that produce approximate solutions.
  Specifically, given a weighted, undirected graph $G$ and a positive integer $k$, we desire to find $k$ disjoint trees within $G$ such that each vertex of $G$ is contained in one of the trees and the weight of the largest tree is as small as possible.
  We are unable to find this problem in the graph partitioning literature, but we show that the problem is NP-complete.
  We then propose two approximation algorithms, one that uses a spectral clustering approach and another that employs a dynamic programming strategy, which produce near-optimal partitions on a family of test graphs.
  We describe these algorithms and analyze their empirical performance.
\end{abstract}

\section{Introduction}

We consider the following graph partitioning problem, motivated by work in computational topology: given an edge-weighted, undirected graph $G$ and a positive integer $k$, we desire to find $k$ subtrees $T_1, \ldots, T_k$ within $G$ such that (1) each vertex of $G$ is contained in $T_i$ for some $i$ and (2) the maximum weight of any $T_i$ is as small as possible.
Although graph partitioning problems have been well-studied, we have been unable to find this particular problem in the literature \cite{GraphSurvey, ClusterSurvey}.
Furthermore, standard variations of graph partitioning problems, such as Euclidean sum-of-squares clustering, are known to be NP-hard \cite{Balanced_k-Means}.
Our particular graph partitioning problem is also NP-hard, but a solution to this problem is important for efficiently parallelizing an expensive computation in multidimensional persistent homology \cite{LesnickWright}.
Thus, we desire an efficient algorithm that gives a near-optimal solution to the graph partitioning problem described above.

\subsection*{Motivation}

The computational topology software RIVET performs an expensive computation that is organized by a graph structure \cite{LesnickWright}. Specifically, the program must compute a ``barcode template'' at each vertex of an undirected, edge-weighted graph $G$, which arises as the dual graph of a certain line arrangement.
In this paper we will avoid the details of the barcode template, but simply regard it as ``data'' that must be computed at each vertex of $G$.
The computation of barcode templates is typically expensive.
However, computations at adjacent vertices are similar, and it is sometimes possible to use the result computed at one vertex to quickly obtain the barcode template at an adjacent vertex. 
Edge weights in the graph give computational cost estimates of this process.
Specifically, if vertex $v_1$ is connected to vertex $v_2$ by an edge of weight $w$, then the time required to obtain the data at $v_2$, given the computed data at $v_1$, is roughly proportional to $w$. 

Thus, RIVET is designed to perform an expensive computation at some starting vertex, and then follow edges of $G$ to each other vertex, obtaining the data at each vertex by updating the data from the previous vertex. Since finding the shortest path that covers all vertices in a graph is NP-hard, RIVET uses an approximation based on a minimum spanning tree (MST) of $G$ \cite{LesnickWright}.

However, it is highly desirable to parallelize the RIVET computation, dividing the computational work among multiple processors. That is, if $k$ processors are available, we would like to partition the vertices of $G$ into $k$ sets of ``nearby'' vertices so that a single processor can compute the data required at all vertices in a single set.
To ensure that the entire computation finishes as quickly as possible, we desire that vertices within a single set can be connected by a tree with edges of low weight, and that the maximal weight of any such tree is as low as possible.
Thus, we arrive at the problem discussed in this paper.

\subsection*{Outline of Paper}

We begin by discussing some mathematical background for this problem, including a proof that the problem is NP-complete, in \Cref{Background}.
We then discuss two approximation algorithms that we have implemented: a spectral algorithm in \Cref{SpectralApproach} and a dynamic programming (DP) algorithm in \Cref{DPApproach}.
\Cref{EmpiricalData} presents some empirical data about the performance of these algorithms on a family of test graphs, and \Cref{Conclusion} offers some conclusions and directions for future work.

\section{Mathematical Background}\label{Background}

\subsection*{Terminology}\label{terminology}

We begin by reviewing some standard terminology from graph theory, and then formally introduce the concept of a minimum spanning $k$-forest.

A \emph{graph} $G$ is an ordered pair $(V,E)$, where $V$ is a set of vertices and $E$ is a set of edges, each of which connects two vertices. Formally, an edge is an unordered pair of distinct vertices in $V$, and each such pair occurs at most once in $E$.  
We sometimes say that the graph is \emph{undirected} (since the edges are unordered pairs) and that the graph is \emph{simple} (since each edge occurs at most once).

A \emph{subgraph} of $G$ is a pair $(V',E')$, with $V' \subseteq V$ and $E' \subseteq E$, that is itself a graph. The \emph{degree} of a vertex $v$ is the number of edges that contain $v$.

A graph is \emph{weighted} if each edge is assigned a number, which is referred to as the ``weight'' of the edge. 
For our purposes, we restrict edge weights to nonnegative integers.
Let $w(G)$ denote the weight of graph $G$, which we define to be the sum of the weights of all edges in the graph.

A \emph{cycle} in a graph is an ordered set of vertices $\{v_1, v_2, \ldots, v_j\}$ such that the graph contains an edge $\{v_i, v_{i+1}\}$ for all $i \in \{1, \ldots, j-1\}$ and also an edge $\{v_j, v_1\}$.
A \emph{tree} is a connected graph with no cycles.

A \emph{spanning tree} of a graph $G$ is a subgraph of $G$ that has no cycles (and is thus a tree), but contains all vertices of $G$.
A \emph{minimum spanning tree} (MST) of $G$ is a spanning tree that has minimum weight among all spanning trees of $G$.
Any weighted graph $G$ has one or more minimum spanning trees.
Minimum spanning trees can be found by various algorithms, such as Kruskal's algorithm, which runs in $O(|E| \log |E|)$ time \cite{AlgText}.

We now define a minimum spanning $k$-forest of $G$, which is the central object of study in this paper.

\begin{definition}\label{minSpanForest}
	A \textbf{spanning $k$-forest} of $G$ is a collection of $k$ trees, $T_1, \ldots, T_k$, each a subgraph of $G$, such that each vertex of $G$ is contained in $T_i$ for some $i \in \{1, \ldots, k\}$.
		
    A \textbf{minimum spanning $k$-forest} of $G$ is a spanning $k$-forest such that the quantity $\max\{w(T_1), \ldots, w(T_k)\}$ is minimum among all spanning $k$-forests of $G$.
    The quantity $\max\{w(T_1), \ldots, w(T_k)\}$ is the \textbf{weight} of the minimum spanning $k$-forest.
\end{definition}

Note that a particular graph may have \emph{many} minimum spanning $k$-forests, since there may be many spanning $k$-forests that contain ``heaviest'' trees of the same weight.

\subsection*{Problem Statement}

Given \Cref{minSpanForest}, our problem can be stated as follows: 

\begin{problem}
	Given a graph $G$ and an integer $k > 1$, find a minimum spanning $k$-forest of $G$.
\end{problem}

In the computational setting described in the Introduction, each tree of the minimum spanning $k$-forest gives a set of vertices at which a single processor will compute the required data, along with edges connecting the vertices used in the computational process.
The definition of minimum spanning $k$-forest minimizes the weight of the ``heaviest'' tree, which ensures that the longest-running processor finishes as quickly as possible.

\subsection*{Connection to Minimum Spanning Trees}

It may seem that a minimum spanning $k$-forest of $G$ should be related to one or more minimum spanning trees of $G$.
While minimum spanning $k$-forests and minimum spanning trees are similar concepts, the minimality criterion of a spanning $k$-forest makes its construction much more difficult than that of a MST.
Specifically, the minimality criterion of a spanning $k$-forest involves only the maximum-weight tree of the forest, not the total weight of all trees.

Given any spanning tree $T$ of a graph $G$, one can simply remove any $k-1$ edges from $T$ to produce a spanning $k$-forest of $G$.
However, if $T$ is a minimum spanning tree of $G$, it might \emph{not} be possible to obtain a minimum spanning $k$-forest of $G$ by removing edges from $T$, as \Cref{tree_example} shows.

    \begin{figure}[ht]
      \begin{center}
        \begin{tikzpicture}[scale=1.1,dot/.style={draw,circle,minimum size=1.4mm,inner sep=0pt,outer sep=0pt,fill}]
  \coordinate [dot] (A) at (0,0);
  \coordinate [dot] (B) at (1,0);
  \coordinate [dot] (C) at (2,0);
  \coordinate [dot] (D) at (3,0);
  \coordinate [dot] (E) at (4,0);
  \coordinate [dot] (F) at (2,-1);
  \coordinate [dot] (G) at (3,-1);
    
  \draw (A) -- node[midway,above] {$1$} (B) -- node[midway,above] {$1$} (C) -- node[midway,above] {$1$} (D) -- node[midway,above] {$1$} (E);
  \draw (C) -- node[midway,left] {$1$} (F) -- node[midway,above] {$1$} (G) -- node[midway,right] {$2$} (D);
  
  \node at (1,-.5) {$G$};
\end{tikzpicture}
\hspace{36pt}
\begin{tikzpicture}[scale=1.1,dot/.style={draw,circle,minimum size=1.4mm,inner sep=0pt,outer sep=0pt,fill}]
  \coordinate [dot] (A) at (0,0);
  \coordinate [dot] (B) at (1,0);
  \coordinate [dot] (C) at (2,0);
  \coordinate [dot] (D) at (3,0);
  \coordinate [dot] (E) at (4,0);
  \coordinate [dot] (F) at (2,-1);
  \coordinate [dot] (G) at (3,-1);
    
  \draw (A) -- node[midway,above] {$1$} (B) -- node[midway,above] {$1$} (C) -- node[midway,left] {$1$} (F);
  \draw (G) -- node[midway,right] {$2$} (D) -- node[midway,above] {$1$} (E);
  
  \draw[gray,dashed] (C) -- (D);
  \draw[gray,dashed] (F) -- (G);
  
  \node at (1,-.5) {$T_1$};
  \node at (3.9,-.5) {$T_2$};
\end{tikzpicture}
      \end{center}
        \caption{The only MST of graph $G$ (left) contains all edges of $G$ except for the edge of weight $2$. The minimum spanning $2$-forest of $G$ consists of two trees $T_1$ and $T_2$ (right), and is not contained in the MST. }
        \label{tree_example}
    \end{figure}
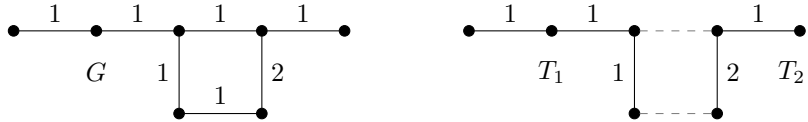

\subsection*{NP-Completeness} \label{NP-Complete}

The Main Problem is NP-complete. 
We give a proof for the decision version of the Main Problem: given a graph $G$ and a positive number $W$, decide whether $G$ has a minimum spanning $2$-forest of weight $W$.
NP-completeness of the Main Problem then follows.

\begin{theorem}\label{NP_theorem}
    The problem of deciding whether a graph $G$ has a minimum spanning $2$-forest of weight $W$ is NP-complete.
\end{theorem}
\begin{proof}
	We show that any algorithm that can solve the problem of deciding whether a graph $G$ has a minimum spanning $2$-forest of weight $W$ (for any $W$) can also decide whether a multiset of positive integers can be partitioned into two subsets with the same sum, which is an NP-hard problem.
    
    Let $x_1, x_2, \ldots, x_m$ be positive integers whose sum is $2W$.
    Consider the graph in \Cref{graphNP}.
    The graph contains $3m$ edges. Of these, $m$ are drawn vertically, with weights $x_1, \ldots, x_m$. The remaining edges have weight $0$.
    
    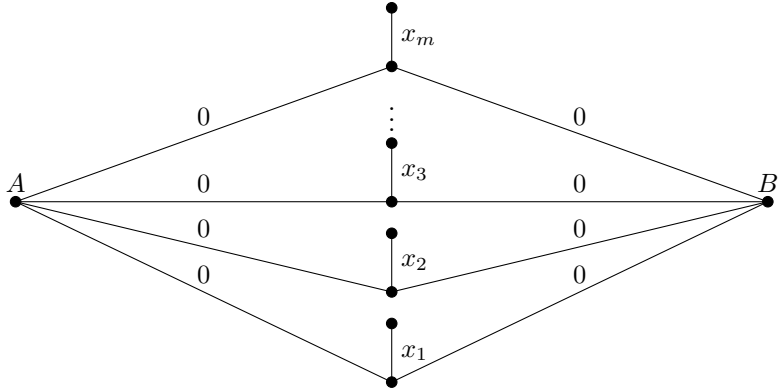
\begin{figure}[ht]
      \begin{center}
        \begin{tikzpicture}[yscale=0.6,dot/.style={draw,circle,minimum size=1.4mm,inner sep=0pt,outer sep=0pt,fill}]
    \coordinate [dot] (S) at (0,0);
    \node[above] at (S) {$A$};
    \coordinate [dot] (T) at (10,0);
    \node[above] at (T) {$B$};
    
    \draw (5,-4) coordinate [dot] (X1) -- node [midway,right] {$x_1$} +(0,1.3) coordinate [dot];
    \draw (S) -- node [midway, above] {$0$} (X1) -- node [midway, above] {$0$} (T);
    
    \draw (5,-2) coordinate [dot] (X2) -- node [midway,right] {$x_2$} +(0,1.3) coordinate [dot];
    \draw (S) -- node [midway, above] {$0$} (X2) -- node [midway, above] {$0$} (T);
    
    \draw (5,0) coordinate [dot] (X3) {} -- node [midway,right] {$x_3$} +(0,1.3) coordinate [dot];
    \draw (S) -- node [midway, above] {$0$} (X3) -- node [midway, above] {$0$} (T);
    
    \node at (5,2) {$\vdots$};
    
    \draw (5,3) coordinate [dot] (X4) {} -- node [midway,right] {$x_m$} +(0,1.3) coordinate [dot];
    \draw (S) -- node [midway, above] {$0$} (X4) -- node [midway, above] {$0$} (T);
\end{tikzpicture}
      \end{center}
        \caption{Graph $G$ used in the proof of \Cref{NP_theorem}.}
        \label{graphNP}
    \end{figure}
    
    We show that a minimum spanning $2$-forest of $G$ gives a partition of $\{ x_1, \ldots, x_m \}$.
    First, observe that since $G$ contains $2m+2$ vertices, a minimum spanning $2$-forest $F$ of $G$ must contain exactly $2m$ edges that form no cycles.

    All of the edges with weights $x_1, \ldots, x_m$ must be in $F$. Otherwise, $F$ would contain an isolated vertex, but a spanning $2$-forest that contains an isolated vertex cannot be minimum for $m > 1$.
    
    Thus, $F$ must also contain $m$ of the edges of weight $0$.
    Together with the edges of weight $x_i$, these edges form two trees.
    Except in the trivial case where some $x_i = W$, it must be that vertices $A$ and $B$ are in separate trees.

	Let $T_A$ be the tree containing $A$, and let $T_B$ be the tree containing $B$.
    Then $w(T_A) = x_{i_1} + \cdots + x_{i_p}$ and $w(T_B) = x_{j_1} + \cdots + x_{j_q}$, where $\{i_1, \ldots, i_p\}$ and $\{j_1, \ldots, j_q\}$ form a partition of $\{1, \ldots, m\}$.
    Moreover, $F$ is minimum if and only if the difference between $w(T_A)$ and $w(T_B)$ is as small as possible.
    This difference is zero if and only if $G$ has a minimum spanning $2$ forest of weight $W$, which is the case if and only if it is possible to partition the set $\{x_1, x_2, \ldots, x_m\}$ into two subsets, both of which sum to $W$.
    
    Therefore, an algorithm that can decide whether there exists a minimum spanning $2$-forest of $G$ of weight $W$ also solves the partition problem, which implies that the decision version of the Main Problem is NP-hard. 
    Since this problem is clearly also NP, it is NP-complete.
\end{proof}

Given that the Main Problem is NP-complete, we desire to find an efficient algorithm that produces an approximate solution to this problem. 
Roughly speaking, an ``approximate solution'' to the Main Problem is a spanning $k$-forest of $G$ such that the max-weight tree of the forest is not much larger than $\frac{w}{k}$, where $w$ is the weight of a minimum spanning tree of $G$.

We now present two such algorithms. 
The first involves a \emph{spectral clustering} technique that is similar to the ``normalized cuts'' method of image segmentation.
The second algorithm uses a dynamic programming (DP) approach to partition a minimum spanning tree of $G$ into $k$ parts.
These algorithms perform well in practice, returning approximate solutions to the Main Problem in our tests (see \Cref{EmpiricalData} for empirical data).

\section{Spectral Algorithm}\label{SpectralApproach}
	
It seems natural to recast the Main Problem as a clustering problem: we wish to find $k$ clusters of the vertices of $G$ such that the vertices within each cluster are connected by relatively low-weight edges, and the edges that connect different clusters have relatively higher weights. Thus, we can find an approximate solution by using a \emph{spectral clustering} algorithm to obtain such clusters of the vertices in $G$, and then finding a minimum spanning tree within each cluster. Specifically, our algorithm is based on the normalized spectral clustering of Shi and Malik, commonly known as the \emph{normalized cuts} algorithm \cite{NormCuts}.

\subsection*{Step 1: Calculating Similarity}

Spectral clustering requires a similarity metric: a real-valued function on pairs of vertices that takes large values on pairs of similar vertices and values close to zero on pairs of dissimilar vertices.
In our case, two vertices are ``similar'' if they are connected by a low-weight path in $G$.
Specifically, suppose $V = \{v_1, v_2, \ldots, v_m \}$.
Let $P(i,j)$ be the weight of the minimum-weight (``shortest'') path between vertices $v_i$ and $v_j$ in $G$, and let $L$ be the maximum weight of any shortest path between any two vertices in $G$.
We then define the function $\omega(i,j)$ as
    \begin{equation*}
	    \omega(i,j) = L - P(i,j)
    \end{equation*}
The function $\omega$ gives the notion of similarity between vertices which we use for our spectral clustering.
Let $W$ be the $m\times m$ matrix of similarity values; that is, $W(i,j) = \omega(i,j)$.
Note that $W$ is a symmetric matrix.

\subsection*{Step 2: Finding an Eigenvector}
We now introduce some terminology to explain the normalized cuts algorithm.
Suppose we desire to partition the vertex set $V$ of $G$ into two nonempty disjoint subsets $A$ and $B$ (with $A \cup B = V$).
We define a \emph{cut} in graph $G$ as follows:
    \begin{equation*}
        \cut(A,B)= \sum_{u\in A, v\in B} \omega(u,v).
    \end{equation*}

Choosing $A$ and $B$ to minimize $\cut(A,B)$ may result in a lopsided partition: for example, $A$ might consist of a single vertex. To avoid this, we use a normalized cut. 
We first define a measure of association between a vertex subset $A$ and the entire vertex set $V$: $\assoc(A,V) = \sum_{u\in A, t \in V} \omega(u,t)$.
This allows us to \emph{normalize} a graph cut, which takes into account the association between each subset and the entire vertex set.
Concretely, a \emph{normalized cut} is defined:
    \begin{equation}
        \Ncut(A,B)= \frac{\cut(A,B)}{\assoc(A,V)} + \frac{\cut(A,B)}{\assoc(B,V)}.
    \end{equation}
Similarly, we define \emph{normalized association}:
    \begin{equation}
        \Nassoc(A,B) = \frac{\assoc(A,A)}{\assoc(A,V)} + \frac{\assoc(B,B)}{\assoc(B,V)}.
    \end{equation}
Shi and Malik prove that minimizing $\Ncut$ is equivalent to maximizing $\Nassoc$. 
While this problem is NP-hard, an approximate solution can be obtained by solving a generalized eigenvalue problem \cite{NormCuts}.

With matrix $W$ as defined above, let $D$ be a $m \times m$ diagonal matrix where $D_{ii}$ is the sum of elements in row $W_i$. Consider the equation
   \begin{equation}\label{eigen}
        (D-W)y = \lambda D y.
    \end{equation}
Solving equation \eqref{eigen} for $\lambda$ and $y$ is a symmetric generalized eigenvalue problem.
A vector $y$ corresponding to the smallest positive eigenvalue $\lambda$ of \eqref{eigen} can be used to obtain subsets $A$ and $B$ that approximately minimize $\Ncut(A,B)$.
        
\subsection*{Step 3: Constructing the subgraphs}
The entries in eigenvector $y$ indicate the partition of $V$ as follows: if $y_i > 0$, then $v_i \in A$; otherwise $v_i \in B$.
We then construct a subgraph $G_A$ with vertex set $A$ by including all edges in $G$ that connect two vertices within $A$; we similarly construct a subgraph $G_B$ with vertex set $B$. If $k=2$, then we find a MST within each subgraph; this provides an approximate solution to the Main Problem. An issue arises when a value $y_i$ is very close to $0$, as $v_i$ could be adjacent to a cut edge and may be placed in a vertex set $A$ such that no edge in $G$ connects $v_i$ to any vertex in $A$. If this occurs then we run Dijkstra's Algorithm to connect $v_i$ with the nearest vertex $v_j \in A$ adding all the vertices in the path to $A$.

\subsection*{More than two blocks}

If $k$, the desired number of trees, is a power of $2$, then we simply apply the process described in Steps 1 through 3 recursively on the subgraphs $G_A$ and $G_B$, until we have obtained the desired number of subgraphs. We then find a MST within each subgraph. This strategy performs very well in practice when $k$ is a power of $2$, as seen in \Cref{EmpiricalData}.

Unfortunately, the spectral algorithm does not produce desirable solutions to the Main Problem if $k$ is not a power of $2$. 
In other applications, spectral clustering can be used to find more blocks than desired in a partition, then blocks are combined as necessary (e.g., via $k$-means clustering) to obtain the desired number of blocks \cite{NormCuts}. However, this approach is problematic in our case since the algorithm yields subtrees of roughly equal weight; combining two or more subtrees is likely to produce a new subtree much larger than the others. 
We are not aware of a method for adapting the spectral algorithm to obtain an approximately minimum spanning $k$-forest when $k$ is not a power of $2$.

\subsection*{Summary}

In short, the spectral approach converts the graph into two matrices based on a defined similarity metric. These matrices are used to solve a generalized eigenvalue problem to obtain an eigenvector that determines a partition of the vertex set into two blocks. Subgraphs are then constructed by connecting every vertex in a given block to vertices in the same block with the edges in the original graph, which partitions the graph into two blocks. This process is repeated to obtain the number of desired blocks.

\subsection*{Computational Complexity}

Computation of $\omega(u,v)$ for all pairs of vertices $u$ and $v$ requires computing the shortest paths between all pairs of vertices. 
Our implementation computes all shortest paths using the Floyd-Warshall algorithm, which runs in $O(|V|^3)$ time \cite{AlgText}. 
More efficient algorithms, such as the Pettie and Ramachandran algorithm, are available \cite{PettieRamachandran}.

The eigenvector can also be found via various algorithms.
Since these methods involve convergence of matrices under iterations of the algorithm, stating precise runtime bounds is difficult.
We use the software LAPACK (Linear Algebra PACKage) to find the eigenvector. Specifically, we call the LAPACK routine {\sf DSPGVX}, which uses generalized Schur decomposition (also called QZ decomposition) \cite{LAPACK}. 

\section{Dynamic Programming Algorithm}\label{DPApproach}

The dynamic programming approach partitions a minimum spanning tree of $G$ into $k$ pieces, obtaining an approximate minimum spanning $k$-forest of $G$.
Thus, the algorithm involves two steps.
First, we find a MST of $G$.
For this step, we developed a modified Kruskal's algorithm to reduce the branching of the MST, which makes our second step more effective.
In the second step, we apply a dynamic programming strategy to the MST, removing $k - 1$ edges to obtain a spanning $k$-forest of $G$.

\subsection*{Step 1: Modified Kruskal's Algorithm}

Kruskal's algorithm is a standard algorithm for finding a MST: the algorithm begins by initializing a forest $F$ that contains each vertex in $G$ as a separate tree, and a set $S$ containing all edges of $G$. The algorithm considers all edges in $S$ in order of increasing weight. If an edge $e$ connects two different trees in $F$, then the algorithm replaces these two trees with a single tree formed by their union together with $e$. Any edge $e$ in $S$ that does not connect two different trees would create a cycle, and is removed from $S$. If $G$ is connected, then the algorithm stops when $F$ contains a single (spanning) tree.

We use a modified version of Kruskal's algorithm to find a MST of $G$ with relatively low degree. Specifically, we modify Kruskal's algorithm by refining the ordering of edges in $S$.
When $S$ contains two or more edges of equal weight, each of which connects two trees in $F$, we select the one that minimizes the degree of the nodes that would be connected in the resulting tree. Our algorithm appears in \cref{alg:modKruskal}.

This strategy of giving priority to edges that connect low-degree vertices at each step of Kruskal's algorithm is not guaranteed to always produce the minimum-degree MST, but we have found that it produces relatively low-degree MSTs in practice.

  \begin{algorithm}[h]
    \caption{Modified Kruskal's Algorithm}\label{alg:modKruskal}
    \footnotesize
    \begin{algorithmic}[1]
      \Require Connected edge-weighted graph $G = (V, E)$
      \Ensure Minimum spanning tree $T$ of $G$ of low degree
      \State Initialize $F$ to contain each $v \in V$ as a tree
      \State $S \leftarrow E$
      \While{$F$ contains more than $1$ tree}
        \State Remove from $S$ all edges that do not connect two trees in $F$
        \State $S' \leftarrow$ set of lowest-weight edges in $S$
        \State $e \leftarrow $ null; $d \leftarrow \infty$; $T_1 \leftarrow$ null; $T_2 \leftarrow$ null
        \ForEach{$e' = (v_1,v_2) \in S'$}
          \State $T'_1 \leftarrow$ tree in $F$ containing $v_1$
          \State $T'_2 \leftarrow$ tree in $F$ containing $v_2$
          \State $d' \leftarrow$ max degree of $v_1$ and $v_2$ in $T'_1 \cup e' \cup T'_2$
          \If{$d' < d$}
            \State $e \leftarrow e'$; $d \leftarrow d'$; $T_1 \leftarrow T'_1$; $T_2 \leftarrow T'_2$
          \EndIf
        \EndFor
        \State Replace $T_1$ and $T_2$ with $T_1 \cup e \cup T_2$ in $F$
        \State Remove $e$ from $S$
      \EndWhile
    \end{algorithmic}
  \end{algorithm}

\subsection*{Step 2: Dynamic Programming Strategy}

Let $T$ be the MST of $G$ found by \Cref{alg:modKruskal} above.
We now use a dynamic programming process to find $k-1$ edges that we can ``cut'' from $T$ to produce a spanning $k$-forest of $G$.
Our algorithm appears in \cref{alg:DPCuts}.

We first find a path $P$ in $T$ of maximal weight; let $s$ and $t$ be vertices at either end of this path.
Let $e_1, \ldots, e_n$ denote the edges along path $P$ from $s$ to $t$.
We will regard these edges as arranged horizontally with $s$ on the left and $t$ on the right, as in \cref{dp_example}.
For $1 \le i \le n$, let $t_i$ be the weight of the portion of $T$ to the left of $e_i$, including $e_i$.
For the tree depicted in \cref{dp_example}, $t_1 = 2$, $t_2 = 6$, and so on.
For notational convenience below, let $t_0 = 0$ and $t_{n+1} = t_n$, which is the weight of $T$.

  \begin{algorithm}[h]
    \caption{Dynamic Programming Partitioning}\label{alg:DPCuts}
    \footnotesize
    \begin{algorithmic}[1]
      \Require Tree $T$, $k = $ number of parts desired 
      \Ensure Cut set $S$ of edges to remove from $T$
      \State Find the longest path $P$ in $T$; let $P = \{e_1, e_2, \ldots, e_n\}$
      \For{$i$ from $1$ to $n$}
      	\State $t_i \leftarrow$ weight of $T$ to the left of and including edge $e_i$
      \EndFor
      \State $t_0 \leftarrow 0$; $t_{n+1} \leftarrow t_n$
      \State $r_0 \leftarrow 0$
      \State $p_0 \leftarrow 0$
      \For{$i$ from $1$ to $n+1$}
        \State $\displaystyle p_i \leftarrow \min_{0 \le j < i} \left( \left| (t_i - w(e_i) - t_j) - \frac{w(T) - r_j - w(e_i)}{k} \right| + p_j \right)$ 
        \State $q_i \leftarrow j$ that minimizes the quantity above
        \State $r_i \leftarrow r_{q_i} + e_i$
      \EndFor
      \State Initialize the cut set $S$ to the empty set
      \State $c \leftarrow q_{n+1}$ \Comment{Index of first edge to cut}
      \While{$c \ne 0$}
         \State Add $e_c$ to the set $S$
         \State $c \leftarrow q_c$ \Comment{Index of next edge to cut}
      \EndWhile
    \end{algorithmic}
  \end{algorithm}
  
  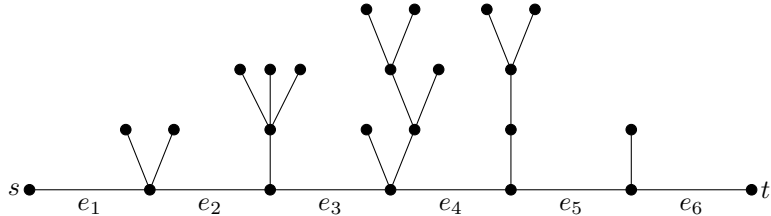
\begin{figure}[ht]
    \begin{center}
      \begin{tikzpicture}[scale=0.8,dot/.style={draw,circle,minimum size=1.4mm,inner sep=0pt,outer sep=0pt,fill}]
  \coordinate [dot] (v0) at (0,0);
  \coordinate [dot] (v1) at (2,0);
  \coordinate [dot] (v2) at (4,0);
  \coordinate [dot] (v3) at (6,0);
  \coordinate [dot] (v4) at (8,0);
  \coordinate [dot] (v5) at (10,0);
  \coordinate [dot] (v6) at (12,0);
  
  \draw (v0) -- node[midway,below] {$e_1$} (v1) -- node[midway,below] {$e_2$} (v2) -- node[midway,below] {$e_3$} (v3) -- node[midway,below] {$e_4$} (v4) -- node[midway,below] {$e_5$} (v5) -- node[midway,below] {$e_6$} (v6);
  
  \node[left] at (v0) {$s$};
  \node[right] at (v6) {$t$};
  
  \coordinate [dot] (w10) at (1.6,1);
  \coordinate [dot] (w11) at (2.4,1);
  \draw (w10) -- (v1) -- (w11);
  
  \coordinate [dot] (w20) at (4,1);
  \coordinate [dot] (w21) at (3.5,2);
  \coordinate [dot] (w22) at (4,2);
  \coordinate [dot] (w23) at (4.5,2);
  \draw (v2) -- (w20) -- (w21);
  \draw (w22) -- (w20) -- (w23);
  
  \coordinate [dot] (w30) at (5.6,1);
  \coordinate [dot] (w31) at (6.4,1);
  \coordinate [dot] (w32) at (6,2);
  \coordinate [dot] (w33) at (6.8,2);
  \coordinate [dot] (w34) at (5.6,3);
  \coordinate [dot] (w35) at (6.4,3);
  \draw (w30) -- (v3) -- (w31);
  \draw (w32) -- (w31) -- (w33);
  \draw (w34) -- (w32) -- (w35);
   
  \coordinate [dot] (w40) at (8,1);
  \coordinate [dot] (w41) at (8,2);
  \coordinate [dot] (w42) at (7.6,3);
  \coordinate [dot] (w43) at (8.4,3);
  \draw (v4) -- (w40) -- (w41);
  \draw (w42) -- (w41) -- (w43);

  \coordinate [dot] (w5) at (10,1);
  \draw (v5) -- (w5);
\end{tikzpicture}
    \end{center}
      \caption{Sample tree $T$ to illustrate the dynamic programming algorithm. Edges $e_1, \ldots, e_6$ comprise path $P$ and each have weight $2$; all other edges have weight $1$.}
      \label{dp_example}
  \end{figure}
  
We iterate over each edge in path $P$ (from $e_1$ to $e_n$) and calculate two quantities associated with each edge: a ``penalty'' and a ``discount.'' 
For edge $e_i$, the penalty is denoted $p_i$ and the discount is denoted $r_i$.
Intuitively, the penalty $p_i$ is a positive number that is close to zero if it is possible to cut edge $e_i$, along with some edges from the set $\{ e_1, \ldots, e_{i-1} \}$, and obtain a set of trees whose weights are near $\frac{w(T) - r_i}{k}$.
The discount $r_i$ is the sum of the weight of edge $e_i$ and weights of edges to the left of $e_i$ that are cut to obtain the lowest penalty.
  
For notational convenience, let $p_0 = 0$.
Then, for $i = 1, \ldots, n+1$, the penalty $p_i$ is defined:
  \begin{equation}\label{penalty}
    p_i = \min_{0 \le j < i} \left( \left| (t_i - w(e_i) - t_j) - \frac{w(T) - r_j - w(e_i)}{k} \right| + p_j \right).
  \end{equation}
In equation \eqref{penalty}, the quantity $(t_i - w(e_i) - t_j)$ is the weight of the subtree \emph{between} edges $e_j$ and $e_i$, and the quantity $\frac{w(P) - r_j - w(e_i)}{k}$ is the weight of each subtree \emph{if it were possible} to obtain $k$ subtrees by removing edge $e_i$ and the edges counted in the discount $r_j$.
Furthermore, to make sense of equation \eqref{penalty}, we define $w(e_{n+1}) = 0$; intuitively, we may regard $e_{n+1}$ as an edge of weight $0$ to the right of $e_n$, which is \emph{always} the last edge that our algorithm cuts from $T$.
Thus, penalty $p_{n+1}$ is the total penalty associated with the best possible cut set for the entire tree $T$.  

For each $i$ from $1$ to $n+1$, the algorithm stores not only $p_i$ and $r_i$, but also the index $j$ that minimizes the quantity in equation \eqref{penalty} above; let $q_i$ be this index.
Thus we can compute the discount $r_i = r_{q_i} + w(e_i)$, which is used in the calculation of penalties $p_\ell$ for $\ell > i$.
Moreover, at the conclusion of the dynamic programming process, the $q_i$ allow us to recover the edges that we will cut from $T$ to minimize the penalty $p_{n+1}$.

Specifically, to obtain the set of edges that we must cut, we backtrack from along $P$ from $e_n$ to $e_1$ (right to left in \cref{dp_example}).
The index rightmost edge to cut is given by $q_{n+1}$; the index of the next edge to cut is given by $q_{q_{n+1}}$, and so on. 
We have found the entire cut set when we encounter some $q_{j} = 0$.
Removing the edges in the cut set from $T$, we obtain a spanning forest $F$.

\begin{example}
  Running \cref{alg:DPCuts} on the tree $T$ depicted in \cref{dp_example} with $k=3$, we obtain the values displayed in the following table.
  \begin{center}\begin{tabular}{ccccc}
    \toprule
    $i$ & $t_i$ & $p_i$ & $q_i$ & $r_i$ \\
    \midrule
    $0$ & $0$ & $0$ &   & $0$ \\
    $1$ & $2$ & $9$ & $0$ & $2$ \\
    $2$ & $6$ & $5$ & $0$ & $2$ \\
    $3$ & $12$ & $1$ & $0$ & $2$ \\
    $4$ & $20$ & $10/3$ & $3$ & $4$ \\
    $5$ & $26$ & $14/3$ & $3$ & $4$ \\
    $6$ & $29$ & $4$ & $4$ & $6$ \\
    $7$ & $29$ & $4$ & $4$ & $4$ \\
    \bottomrule
  \end{tabular}\end{center}
  The algorithm proceeds along path $P$ from left to right.
  Cutting edge $e_1$ would result in an isolated vertex (weight $0$), while the rest of the tree has weight $27$, so penalty $p_1 = \left| 0 - \frac{27}{3} \right| = 9$ and discount $r_1 = 2$. 
  If edge $e_2$ is cut, then it is best not to cut $e_1$, but to obtain a subtree of weight $4$, so penalty $p_2 = \left| 4 - \frac{27}{3} \right| = 5$ and discount $r_1 = 2$.
  Similarly, if edge $e_3$ is cut, then the penalty $p_3 = 1$ is obtained by cutting neither $e_2$ nor $e_1$.
  However, if edge $e_4$ is cut, then the lowest penalty is obtained by also cutting $e_3$; the subtree between $e_3$ and $e_4$ has weight $6$, and with two edges cut, penalty $p_4 = \left| 6 - \frac{25}{3} \right| + 1 = \frac{10}{3}$ and discount $r+4 = 4$.
  The other penalties and discounts are obtained similarly.
  
  The algorithm then backtracks to find the cut set.
  Since $q_7 = 4$, edge $e_4$ will be cut. 
  Next, $q_4 = 3$ implies that edge $e_3$ will be cut.
  However, $q_3 = 0$, so no other edges will be cut.
  Therefore, edges $e_3$ and $e_4$ are cut, resulting in a spanning forest $F$ containing trees of weight $10$, $6$, and $9$.
\end{example}

While the penalty calculation in \Cref{alg:DPCuts} gives preference to cuts that result in subtrees of weight close to $\frac{w(T)}{k}$, the algorithm is not \emph{guaranteed} to return a cut set of size $k-1$. Thus, $F$ is not guaranteed to contain $k$ trees.
In practice, the algorithm often does return a cut set of size $k-1$.
However, if the algorithm returns a cut set of size less than $k-1$, we find additional edges to cut by repeatedly applying \Cref{alg:DPCuts} to the largest remaining subtree. 
If the algorithm returns a cut set of size greater than $k-1$, then we join the pair of smallest subtrees in $F$ that are adjacent in $T$, continuing to do so until $F$ contains exactly $k$ trees.

\subsection*{Computational Complexity}

It is well known that the standard Kruskal's Algorithm runs in $O(|E| \log|E|)$ time, using a union-find data structure to store the forest $F$ \cite{AlgText}.
In our modified algorithm, at each iteration of the while loop in \Cref{alg:modKruskal}, we must identify the edge, among edges of lowest weight that do not create cycles, that connects minimum-degree vertices.
Maintaining a list of vertex degrees in $F$ allows for constant-time vertex degree lookup, but we must still find the minimum among $O(|E|)$ edges at each iteration of the while loop.
Since this loop runs $O(|E|)$ times, the complexity of our modified Kruskal's algorithm is $O(|E|^2)$.

Our dynamic programming algorithm begins by finding the longest path $P$ in $T$.
This requires finding the distance between all pairs of nodes in $T$, which can be accomplished in $O(|V|^2)$ time by a depth-first search of $T$.
We then iterate over every edge in $P$ and perform a constant-time calculation for each previously-considered edge in $P$.
Since the number of edges in $P$ is not more than $|E|$, this requires $O(|E|^2)$ time.
Since $E = V-1$, the overall complexity of the dynamic programming algorithm is $O(|E|^2)$.

\section{Empirical Data}\label{EmpiricalData}

We tested our algorithms, with various values of $k$, on a family of 90 graphs.
Since the motivation for our work arises from the computational topology software RIVET, we chose graphs representative of those that are encountered in RIVET. These graphs are the dual graphs of planar line arrangements \cite{LesnickWright}.

\subsection*{Methodology}

We first generated a collection of $m$ random lines in the plane. 
These lines form the boundaries of a collection of two-dimensional cells ($2$-cells) in the plane; each $2$-cell has positive (possibly infinite) area. 
The $2$-cells determine a graph $G$ (the \emph{dual graph} of the line arrangement) as follows: each $2$-cell is a vertex of $G$, and two vertices are connected by an edge if and only if the corresponding $2$-cells have a common boundary (this common boundary is a segment contained in one of the lines).
For each line $\ell$, we chose a random integer between $1$ and $1000$, which is used to weight all edges in $G$ that connect cells whose common boundary is line $\ell$.

We repeated the above process $90$ times, for choices of $m$ between $5$ and $100$, to generate $90$ graphs.
We then used our algorithms described in \cref{SpectralApproach} and \cref{DPApproach} to produce spanning $k$-forests from each graph, for values of $k$ between $2$ and $32$.

\subsection*{Data from Two Graphs}

We present here data of the minimum spanning forests obtained from two of our graphs, which are representative of what we observed from our experiments.

Graph $G_1$ is the dual graph of an arrangement of 50 lines; this graph contains 1154 vertices and 2213 edges.
A minimum-spanning tree $T_1$ of $G_1$ has weight 325,427.
We tested our algorithms on this graph for each integer $k$ from $2$ to $8$.
However, we only ran the spectral algorithm when $k$ is a power of $2$ because the algorithm does not give desirable results for other choices of $k$, as explained in \cref{SpectralApproach}.

\cref{graphA_data} summarizes the spanning $k$-forests found by our algorithms when run on Graph $G_1$, for integers $k$ from $2$ to $8$. 
Values in the \emph{Heaviest Subtree} column give the weight of the heaviest tree in the spanning $k$-forest found by the algorithm.
The ratio of the weight of the heaviest subtree to the MST weight is given in the \emph{Ratio} column. 
A lower ratio is better; a ratio of $\frac{1}{k}$ would be ideal in the sense that such a ratio would result from a partition of a MST into $k$ subtrees of equal weight, if all removed edges have weight zero.
However, since some edges of $T_1$ are cut and do not appear in any tree of the spanning $k$-forest, the ratio could be less than $\frac{1}{k}$, but this does not occur in \cref{graphA_data}.

\begin{table}
    \begin{center}\begin{tabular}{l l r r r}
    \toprule
    $k$ & Algorithm & Heaviest Subtree & Ratio \\
    \midrule
    2 & DP & 178,945 & 0.550 \\
    2 & Spectral & 166,010 & 0.510 \\
    3 & DP & 109,338 & 0.336 \\
    4 & DP & 93,334 & 0.287 \\
    4 & Spectral & 90,061 & 0.277 \\
    5 & DP & 69,136 & 0.212 \\
    6 & DP & 67,524 & 0.208 \\
    7 & DP & 53,709 & 0.165 \\
    8 & DP & 53,076 & 0.163 \\
    8 & Spectral & 50,755 & 0.156 \\
    \bottomrule
    \end{tabular}\end{center}
    \caption{Empirical data for Graph $G_1$. The MST of $G_1$ has weight 325,427.}
    \label{graphA_data}
\end{table}

Graph $G_2$ is the dual graph of an arrangement of 100 lines, and contains 4442 vertices and 8543 edges.
A minimum-spanning tree $T_2$ of $G_2$ has weight 1,627,441.
We tested our algorithms on this graph, letting $k$ take on powers of $2$ up to $32$. This allowed us to directly compare our algorithms for each value of $k$. 
\Cref{graphB_data} summarizes the spanning $k$-forest found by our algorithms for Graph $G_2$; the meaning of each column is the same as in the previous table.

The results reported in \cref{graphA_data,graphB_data} are typical of our tests.
We note that the reported ratios are close to $\frac{1}{k}$ in each case.
Since the runtime for the RIVET computation is roughly proportional to the weight of the heaviest tree, the reported ratios represent the factors by which we expect the RIVET runtime to improve when the algorithms are incorporated into RIVET.
Unfortunately, we have no guaranteed bound on the difference between the ratio and $\frac{1}{k}$; this is a question for further research.

 \begin{table}
     \begin{center}\begin{tabular}{l l r r r}
     \toprule
     $k$ & Algorithm & Heaviest Subtree & Ratio \\
     \midrule
     2 & DP & 813,748 & 0.500 \\
     2 & Spectral & 837,742 & 0.515 \\
     4 & DP & 522,254 & 0.321 \\
     4 & Spectral & 424,436 & 0.261 \\
     8 & DP & 274,441 & 0.169 \\
     8 & Spectral & 229,936 & 0.141 \\
     16 & DP & 169,941 & 0.104 \\
     16 & Spectral & 128,014 & 0.079 \\
     32 & DP & 72,667 & 0.045 \\
     32 & Spectral & 68,909 & 0.042 \\
     \bottomrule
     \end{tabular}\end{center}
     \caption{Empirical data for Graph $G_2$. The MST of this graph has weight 1,627,441.}
     \label{graphB_data}
 \end{table}

\subsection*{Summary of Experiments}

As a summary of our other tests, \cref{graph_summary} gives statistics from our tests on our $30$ largest graphs, with $k$ taking on powers of $2$, up to $32$.
For each graph, algorithm, and value of $k$, we computed the ratio of the weight of the largest subtree found by the algorithm to the weight of the MST of the graph (as reported in the \emph{Ratio} columns of \Cref{graphA_data,graphB_data}). 
The table gives the average and standard deviation of $30$ such ratios for each algorithm and each $k$. We used a paired $t$-test to compare the ratios achieved by the two algorithms; the table gives the $t$-statistics and 2-sided $p$-values for each $k$.

\begin{table}
    \begin{center}\begin{tabular}{r c c c c c}
    \toprule
    $k$ & $2$ & $4$ & $8$ & $16$ & $32$ \\
    \midrule
    Spectral average & $0.517$ & $0.274$ & $0.149$ & $0.082$ & $0.048$ \\
    Spectral st.\ dev. & $0.011$ & $0.011$ & $0.008$ & $0.008$ & $0.014$ \\
	DP average & $0.530$ & $0.288$ & $0.158$ & $0.084$ & $0.044$ \\
    DP st.\ dev. & $0.049$ & $0.024$ & $0.019$ & $0.008$ & $0.003$ \\
    $t$-statistic & $1.398$ & $3.063$ & $2.681$ & $0.913$ & $-1.674$ \\
    $p$-value & $0.173$ & $0.005$ & $0.012$ & $0.369$ & $0.105$ \\
    \bottomrule
    \end{tabular}\end{center}
    \caption{Summary of tests from $30$ large graphs. For each algorithm and each value of $k$, we give the mean and standard deviation of $30$ ratios computed. We used a paired $t$-test to compare the ratios; the $t$-statistics and 2-sided $p$-values appear here.}
    \label{graph_summary}
\end{table}

Our spectral algorithm performed very well in our tests with $k$ a small power of $2$, consistently producing ratios close to $\frac{1}{k}$.
Our dynamic programming algorithm performed reasonably well in all cases, but generally not quite as well as the spectral algorithm when $k \in \{2, 4, 8, 16\}$.
Notably, for small $k$ the ratios achieved by the spectral algorithm exhibit less variation than those achieved by the DP algorithm, as indicated by the standard deviations in \cref{graph_summary}. 
Furthermore, the difference between the ratios produced by the two tests were most statistically significant for $k=4$ and $k=8$, as shown by the $p$-values in \cref{graph_summary}. 

Interestingly, we noticed that for some graphs, the DP algorithm performs better for certain smaller values of $k$ than for certain larger values of $k$. 
For example, when run on a particular graph, the algorithm returns a smaller largest subtree when $k=7$ than when $k=8$.
This is due to the fact that the initial cutting process does not always produce exactly $k$ subtrees, in which case we cut the largest subtree again.
Of course, we do not expect this algorithm to return \emph{the optimal} partition, but it does seem to return approximately optimal partitions in every case.

\subsection*{Runtime}

The DP algorithm ran much faster than the spectral algorithm in our tests, but the runtime of the spectral algorithm was dominated by the runtime of the Floyd-Warshall algorithm for finding all shortest paths. We report here a few sample runtimes for the graphs discussed above, recorded when testing our algorithms on a typical laptop computer.

For graph $G_1$ with $k=8$, the DP algorithm ran in 0.2 seconds. The spectral algorithm ran in 26.0 seconds; of this time, 24.0 seconds was spent running the Floyd-Warshall algorithm.
For the graph $G_2$ with $k=32$, the DP algorithm ran in 8.3 seconds. The spectral algorithm ran in 1,546.2 seconds, of which 1,431.1 seconds was spent running the Floyd-Warshall algorithm.
Thus, it seems that more than 90\% of the spectral algorithm runtime is spent in the Floyd-Warshall algorithm. As mentioned in \cref{SpectralApproach}, more efficient algorithms for finding all shortest paths are available.

Importantly, the runtime for either of our algorithms is orders of magnitude less than the runtime of the RIVET computation.
Thus, we are confident that the incorporation of either algorithm into RIVET as the basis for a parallelized barcode template computation will result in substantial improvement in the RIVET runtime.
Indeed, the best strategy may be to run both the spectral and DP algorithms, and then to choose the resulting forest of lowest weight for the computation of the barcode templates.

\subsection*{Code}

The code used to generate the synthetic graphs and for the both the spectral and DP algorithms can be found at \url{https://github.com/a3madkour/Minimal-Spanning-Trees}.

\section{Conclusion}\label{Conclusion}

Though the Main Problem is NP-Hard, we have two algorithms that produce approximate solutions. 
Our spectral algorithm performed very well in our tests with $k$ a power of $2$. 
However, as stated in \cref{SpectralApproach}, we are unable to obtain desirable results with the spectral algorithm when $k$ is not a power of $2$. 
Our dynamic programming algorithm performed well in all cases, but less so than the spectral algorithm when $k$ is a small power of $2$.
We believe these algorithms are well-suited for our immediate application in computational topology, though our work has produced various questions for further study.

\subsection*{Future Work}

\paragraph{APX-Hardness.}
We attempted to find an algorithm that can produce a spanning $k$-forest whose weight is within a constant factor of the minimum spanning $k$-forest. However, it is unclear if such an algorithm can run in polynomial time; that is, we are not sure if the Main Problem is APX-hard. 

\paragraph{Spectral Clustering for non powers of 2.}
In \cref{SpectralApproach} we mentioned that there was no obvious way of using the spectral algorithm when $k$ is not a power of $2$.  However, there may yet exist a method of using the spectral approach effectively for values of $k$ that are not powers of $2$. 

\paragraph{DP algorithm questions.}
We suspect that it is possible to improve the performance of our dynamic programming algorithm. For example, it may be possible to modify our ``penalty'' calculation to obtain a better set of subtrees in the first iteration of the algorithm.
Furthermore, when the first iteration of the algorithm does not return exactly $k$ trees, what is the best strategy for making additional cuts? These questions require further research.

\paragraph{Performance guarantees.}
We would like to quantify how close our algorithms come to finding a minimum spanning forest, in a more rigorous way than what we offer in \Cref{EmpiricalData}.
This seems to be a difficult task. 
However, we suspect that progress could be made with the help of more computational experiments and graph-theoretic insight.

\paragraph{Restriction to Trees.}
The first version of this paper conjectured that the Main Problem remains NP-hard when the input graph is a tree, rather than a general graph. However, Vaishali et al.\ give a polynomial time algorithm for solving the Main Problem when the input is a tree \cite{Vaishali}. 
We would like to compare the empirical runtime of Vaishali et al.\ with that of the DP algorithm presented in this paper. We would also like to evaluate the effectiveness of using our modified Kruskal's algorithm, followed by the algorithm of Vaishali et al., for approximating a minimum spanning $k$-forest in a general graph.

\section*{Acknowledgements}

This work was supported by NSF DMS 1606967, the Collaborative Undergraduate Research and Inquiry (CURI) program at St. Olaf College, and the Kay Winger Blair endowment.

\end{document}